\documentclass[a4paper,reqno]{amsart}
\usepackage{amsmath}
\usepackage{amsfonts}
\usepackage{amssymb}
\usepackage{hyperref}
\usepackage{comment}
\usepackage{color}
\usepackage{graphicx}

\newtheorem{theorem}{Theorem}[section]
\newtheorem{lemma}[theorem]{Lemma}

\newtheorem{remark}[theorem]{Remark}

\newcommand{\N}{\mathbb{N}}
\newcommand{\Z}{\mathbb{Z}}
\newcommand{\R}{\mathbb{R}}
\newcommand{\E}{\mathbb{E}}

\newcommand{\D}{\mathcal{D}}
\newcommand{\Q}{\mathcal{Q}}

\newcommand{\PP}{\mathbb{P}}

\newcommand{\iii}{\mathtt{i}}

\newcommand{\norm}[1]{\left\|#1\right\|}

\DeclareMathOperator{\spt}{spt}

\title{Fourier dimension of Mandelbrot cascades on planar curves}
\author{Donggeun Ryou}
\address{DR: Department of Mathematics, Indiana University, Bloomington, IN, USA.}
\email{dryou@iu.edu}
\author{Ville Suomala}
\address{VS: Mathematical Sciences, P.O.Box 8000, FI-90014, University of
Oulu, Finland}
\email{ville.suomala@oulu.fi}
\date{\today}

\thanks{VS was supported by the Research Council of Finland via the project “\emph{Fractals and randomness}”, grant no. 368817. We thank Tuomo Kuusi for pointing our attention to the concentration bounds in \cite{AK} and Tuomas Sahlsten for useful discussions. We are grateful to an anonymous referee for various useful comments.} 

\begin{document}

\begin{abstract}
We consider multifractal Mandelbrot cascades supported on planar $C^2$ curves with nonvanishing curvature and show that their Fourier dimension is as large as possible, i.e., equal to the infimum of the lower pointwise dimension of the measure. 
\end{abstract}

\maketitle

\section{Introduction}

For a finite measure $\eta$ on a Euclidean space 
$\R^d$, we define the Fourier transform of $\eta$ at $\xi\in \R^d$ as
\[
\widehat{\eta}(\xi)=\int_{\R^d} e^{-2\pi i x \cdot\xi} d\eta(x)\,.
\]

While the study of Rajchman measures, i.e. those whose Fourier transform vanishes at infinity, dates back more than a century (see \cite{Lyons1995}), the problem of quantifying the decay rate of the Fourier transform for random measures was initiated in the early 1970s by Mandelbrot \cite{Mandelbrot1974,Mandelbrot1976}. This line of research was further brought up by Kahane \cite{Kahane1993}, who revisited several of the Mandelbrot problems a few decades later.

Although Kaufman’s work in the early 1980s \cite{Kaufman1981} established the existence of random measures with an arbitrary prescribed Fourier dimension, it took almost 50 years before the first nontrivial results appeared for the models that Mandelbrot was interested in.
Before presenting these results, we recall that the \emph{Fourier dimension} of a measure $\eta$, denoted by $\dim_F \eta$, is defined as
\[
\dim_F \eta =\sup\left\{0<s< d\,:\,|\widehat{\eta}(\xi)|=O(|\xi|^{-s/2})\right\}\,.
\]

In \cite{FalconerJin2019}, Falconer and Jin provided the first quantitative lower bounds for the Fourier dimension of Gaussian multiplicative chaos (GMC) measures. Their result concerns the GMC defined on planar domains, and the method relies on the theory of orthogonal projections. The corresponding problem on the real line was addressed by Garban and Vargas \cite{GarbanVargas}, who obtained a lower bound for the Fourier dimension of the GMC on $[0,1]$.
For the related dyadic model of Mandelbrot cascades, the exact value of the Fourier dimension was determined independently by Chen, Han, Qiu, and Wang \cite{CHQW}, and by Chen, Li, and Suomala \cite{CLS}. Subsequently, the Fourier dimension of more general random measures, including GMC defined on the torus $\mathbb{T}^d \subset \mathbb{R}^d$, was established in \cite{LinQiuTan2024,LinQiuTan2025}. 

All the results discussed above concern random measures defined on domains in $\mathbb{R}^d$. In particular, for cascade measures $\nu$ on $[0,1]^d$, the results of \cite{CLS,CHQW} show that the Fourier dimension is, almost surely, given by 
\begin{equation}\label{eq:F_equals_C}
    \dim_F\nu=\min\{2,\dim_2\nu\}\,,
\end{equation} 
where
$\dim_2 \eta$ denotes the correlation dimension of a measure $\eta$, and whose value for the cascade measures has a well-known explicit formula. Recalling the general inequality $\dim_F \eta \le \dim_2 \eta$, valid for all finite Borel measures $\eta$, we thus note that cascade measures on $[0,1]^d$ are quasi-Salem, in the sense that their Fourier dimension is as large as permitted by this inequality.

In this paper, we extend the study of Mandelbrot multiplicative cascades to cascade measures supported on planar curves with nonvanishing curvature. This problem was proposed in \cite[p.~7]{CHQW} in the language of oscillatory integrals, and was also discussed in \cite{CLS}, where a partial and non-optimal result was obtained.
Our main result shows that, for such cascade measures, the Fourier dimension is almost surely equal to the minimum of the multifractal spectrum. More precisely, let $(W_i)_{1\le i\le b}$ be non-negative random variables with unit expectation and superpolynomial tails. Let $\Gamma \subset \mathbb{R}^2$ be a compact $C^2$ curve with nonvanishing curvature, and let $\mu$ be the Mandelbrot cascade constructed from $(W_i)_{i=1}^b$ via a $b$-adic decomposition of $\Gamma$. Denote
$\alpha_{\min}=\inf\{\dim(\mu,x)\,:\,x\in\spt\mu\}$, where $\dim(\eta,x)$
is the pointwise Hausdorff dimension of a measure $\eta$ at $x$. Deferring the detailed definitions and our technical assumptions to Section \ref{sec:prel_model}, we now state the main theorem.
\begin{theorem}\label{thm:intro}
   $\dim_F\mu=\alpha_{\min}$ almost surely on non-extinction.  
\end{theorem}
We note that $\alpha_{\min}$ has a well-known analytic expression in terms of the random variables $(W_i)$.
We also extend the result by establishing bounds on the decay rate of the  spherical $L^p$ averages,
\begin{equation}\label{def:sp}
\sigma_p(\eta)(r) =  \left(\int_{S^1} |\widehat{\eta}(r\theta )|^p d\sigma(\theta) \right)^{1/p}\,,
\end{equation}
for all $1\le p\le\infty$, where $\sigma$ denotes the surface measure on the unit circle. We refer to Theorem~\ref{thm:circ_decay} for the precise statement and note that Theorem \ref{thm:intro} corresponds to the case $p=\infty$, interpreting the definition \eqref{def:sp} as the $L^\infty$ norm if $p=\infty$.

In addition, as a byproduct of the proof of Theorem~\ref{thm:intro}, we obtain a proof of \eqref{eq:F_equals_C}, assuming that 
$(W_i)_i$ satisfies the same hypotheses as in Theorem~\ref{thm:intro}. Although this result is contained in \cite{LinQiuTan2025}, our approach yields a relatively simple proof under the fairly general moment conditions on the random weights. In particular, it applies to lognormal cascades, which constitute the most relevant examples from the viewpoint of applications.

The paper is organized as follows. In Section~\ref{sec:prel_model}, we introduce the model and recall the necessary multifractal tools. In Section~\ref{sec:aux}, we recall the formula for the correlation dimension of the cascades and provide a technical lemma describing the behaviour of certain moment sums of the cascade across different scales. 
Section~\ref{sec:conc} contains our main probabilistic ingredient: a concentration inequality adapted from \cite{AK}, which enables us to improve the methods developed in \cite{CLS, Ryou}. This strategy is implemented in Section~\ref{sec:Other Lemmas}, where we establish the lower bound 
$\dim_F\mu\ge\alpha_{\min}$. The matching upper bound $\dim_F\mu\le\alpha_{\min}$ follows as a corollary of a universal estimate that holds for any finite Borel measure supported on a $C^2$ curve with nonvanishing curvature. The details are provided in Section \ref{sec:ub}. Finally, our result on the decay of 
$\sigma_p(\mu)(r)$ is presented in Section~\ref{sec:spherical}, and our new proof of \eqref{eq:F_equals_C} is found in the Appendix.

\section{Preliminaries}

\subsection{Notation for cascade measures}\label{sec:prel_model}
In relation to the Fourier dimension, we define the correlation dimension, $\dim_2\eta$, of a measure $\eta$. This may be done by partitioning the space into cubical objects of a certain size, computing the $L^2$-sum of their masses, and passing to the limit after a suitable normalisation. For a measure $\eta$ on $\R^d$, we let
\[
\dim_2 \eta =\liminf_{n\rightarrow \infty} \frac{\log \sum_{Q\in \Q_n} \eta(Q)^{2}}{-n}\,,
\]
where $\Q_n$ is the collection of $b$-adic subsquares of side-length $b^{-n}$ and the $\log$ is to base $b$. Here and in what follows, $b\ge 2$ is a fixed integer. It is well known that
$\dim_F\eta\le\dim_2\eta\le\dim_H\eta$ holds for all compactly supported finite measures, where 
$\dim_H\eta$ is the Hausdorff dimension of $\eta$ defined as the supremum of the values $s$ for which $\dim(\eta,x)\ge s$ holds for $\eta$-almost every $x$, and 
\begin{equation*}
    \dim(\eta,x)=\liminf_{ r \rightarrow 0} \frac{\log(\eta(B(x,r))}{\log r}
\end{equation*}
is the pointwise Hausdorff dimension of $\eta$ at $x$.
We first define Mandelbrot multiplicative cascades on $[0,1]^d$ with respect to the base $b$. To that end, let $\Lambda=\{0,\ldots,b-1\}^d$, and given $\iii=(i_1,\ldots,i_n)\in\Lambda^n$, let $x_\iii\in[0,1]^d$ such that $(x_\iii)_j=\sum_{k=1}^n (i_k)_j b^{-k}$ for all $1\le j\le d$, and let $Q_\iii=x_\iii+[0,b^{-n})^d$. Then $\Q_n=\{Q_\iii\,:\,\iii\in\Lambda^n\}$ is the family of half-open $b$-adic subcubes of $[0,1)^d$ of level $n$. If $\iii=(i_1,\ldots, i_n)\in\Lambda^n$, and $1\le j \le n$, we use $\iii|_j=(i_1,\ldots,i_j)$ to denote the subword containing the first $j$ elements.

The cascade is driven by a random vector $W=(W_i)_{i\in\Lambda}$, where the $W_i\ge 0$ are random variables with 
\begin{equation}\label{eq:cascade_gen}
\E\left(\sum_{i\in\Lambda}W_i\right)=b^d\,.
\end{equation}
We attach an independent copy $W_\iii$ of $W$ to each $\iii\in\Lambda^n$, $n\in\N$. For a fixed $\iii=(\iii_j)_{j=1}^n\in\Lambda^n$,
let
\[\nu_n(x)=\prod_{j=1}^n (W_{\iii|_{j-1}})_{\iii_j},\]
for each $x\in Q_\iii$ (where $W_\varnothing=W$), and let $\nu$ be the weak*-limit of the measures $d\nu_n(x)=\nu_n(x)$. This random cascade measure $\nu$ is non-zero with positive probability if and only if the condition 
\begin{equation}\label{eq:KP}
    \sum_{i\in\Lambda}\E(W_i\log W_i)< d b^d
\end{equation}
is satisfied \cite{KahanePeyriere1976}. This subcriticality condition \eqref{eq:KP} is our standing assumption throughout the paper.

In our main result, we consider Mandelbrot multiplicative cascades defined on curves with nonzero curvature. In this curvilinear setting, the cascade measure is the push forward $\nu  \circ \gamma^{-1}$ where $\gamma: [0,1] \rightarrow \R^2$ is a $C
^2$-curve with $\det(\gamma'(t),\gamma''(t))\neq 0$, and $\nu$ is the cascade measure on the unit interval. We use the following notation: let $\Q_n$, $n\ge 0$ be the $b$-adic filtration on the unit interval and let $\D_n = \gamma(\Q_n)$. Let $\mu_n=\nu_n \circ \gamma^{-1}$, $\mu=\nu \circ \gamma^{-1}$, where $(\nu_n)_n$ is the sequence of the cascade measures associated with an initial random variable $W$. Without loss of generality, we assume $|\gamma'|=1$ and denote
\[\int_{\gamma(J)}f(x)\,dx=\int_{J}f(\gamma(t))\,dt\,,\]
so that the integration is with respect to the arc length. Let us also denote $\Gamma=\gamma([0,1])$.

Throughout the paper, we impose the following additional assumptions on $W_i$, $i\in\Lambda$:
\begin{align}
    &\E(W_i)=1\,,\label{eq:unif}\\
    &\E(W_i^p)<\infty\text{ for all }0<p<\infty\,.\label{eq:moments}
\end{align}

Denote for $q\ge 0$, 
\[\tau(q)=dq-\log\left(\sum_{i\in\Lambda}\E(W_i^q)\right)\,.\]

If there exists a value $q$ such that $q \tau'(q) = \tau(q)$, then this value is unique. We denote it by $q_{\max}$. Note that $q\tau'(q) \geq \tau (q)$ if and only if $q \leq q_{\max}$. If $ q\tau'(q) \geq \tau(q)$ for all $q$, we let $q_{\max} =\infty$.

For $p >0$, we define $\widetilde{\tau}(p)$ as follows:
\begin{equation} \label{eq:alpha-W}
\widetilde{\tau}(p) =
\begin{cases}
\tau(p)& \text{ if } \tau'(p)\ge\tau(p)/p\\ 
\frac{p\tau(q_{\max})}{q_{\max}} & \text{ otherwise. }
\end{cases}
\end{equation}
Also, we let 
\[\alpha_{\min}=\tau'(q_{\max})=\frac{\tau(q_{\max})}{q_{\max}}\,,\]
or let $\alpha_{\min}=\lim_{q\to\infty}\tau'(q)$ if $q_{\max} = \infty$. In other words, $\alpha_{\min}=\lim_{p\to\infty}\tfrac{\widetilde{\tau}(p)}{p}$.

Note that 
\[\alpha_{\min}=\inf\{\dim(\nu,x)\,:\,x\in\text{spt } \nu\} = \inf\{\dim(\mu,x)\,:\,x\in\text{spt }\mu\}\] 
almost surely on non-extinction of the measure $\mu$. We refer e.g. to \cite{Molchan1996, Barral2000, Heur} for these and other basic facts about the multifractal analysis of the cascade measures. We will use the standard $O(\cdot)$ notation and also the notation $f\lesssim g$ as a synonym for $f=O(g)$. If $f\lesssim g\lesssim f$, we denote $f\sim g$. If a constant $C$ may depend on a parameter such as $p$, we write $f \lesssim_p g$ meaning that $f \leq C(p) \, g$  where the constant $C(p) $ may depend on $p$. The dependence on implicit constants will be clarified as needed. For $\xi\in\R^d$, we use the familiar notation $|\xi|_\infty = \max\{|\xi_1|, \cdots, |\xi_d| \}$.

\subsection{Auxiliary results for the Mandelbrot cascades}\label{sec:aux}

Throughout this section, we prove auxiliary results for the cascade measures. We note that the results in this section do not depend on the geometry of the support of the cascade, whether it is a cube in $\R^d$ or a curve, and they could be stated purely in symbolic terms via $\Lambda^{\N}$. To cover also the case $d>1$, we state the results for cascades on $[0,1]^d$, but we stress that they also hold for the curvilinear cascades.

For $1 \leq p,q< \infty$ and $1\le j\le n$, we let
$$
S(p,q,j,n) = \left( \sum_{I \in \Q_j} \left( \sum_{J \in \Q_n , J \subset I} \nu_n(J)^q \right)^{p/q} \right)^{1/p}.
$$ If $p=\infty$, we take the $\ell^\infty$ norm over $I \in \mathcal{Q}_j$, i.e.
$$
S(\infty,q,j,n) =  \sup_{I \in \Q_j} \left( \sum_{J \in \Q_n , J \subset I} \nu_n(J)^q \right)^{1/q}.
$$
Similarly, 
\[S(p,\infty,j,n)=\left( \sum_{I \in \Q_j}\left(\sup_{J \in \Q_n , J \subset I}\nu_n(J)\right)^p\right)^{1/p}.\]
If $p=q=\infty$, we define $S(\infty,\infty, j,n):=S(\infty,1,n,n)$.

To estimate $\E(S(p,q,j,n))$, we use the auxiliary random variables
    \begin{equation*}
    W_q=\frac{b^d(W_j^q)_{j\in\Lambda}}{\sum_{j\in\Lambda}\E(W_j^q)}\,,
    \end{equation*}
    and the auxiliary measure defined by setting $\nu_{j,j}(x)=1$, and
    $$
    \nu_{j,n} (x) = \prod_{k=j+1}^n (W_{\iii|_{k-1}})_{\iii_k}\,,
    $$
    if $n>j$.
    For each $I \in \Q_j$, we define a sequence 
    \[Y_{j,n}(q,I):=b^{ndq}\left(\sum_{i\in\Lambda}\E(W_i^q)\right)^{-(n-j)}\sum_{J\in\Q_n, J \subseteq I }\nu_{j,n}(J)^q\,.\]
    Then, $Y_{j,n}(q,I) $ yields the total measure of $I$ for a cascade measure on the interval $I$ generated by $W_q$.

We next provide a variant of a classical result of Kahane and Peyri\`ere \cite{KahanePeyriere1976} on the $L^q$-boundedness of $\nu_n([0,1]^d)$.

\begin{lemma}\label{lem:E_Y_jn}
    Suppose that $1\leq q<p < q_{\max}$. If $1 < p/q \leq 2$, for any $1 \leq j \leq n $ and  $I \in \Q_j$, we have
    
    \begin{equation}\label{eq:E_Y_1}
    \E(Y_{j,n }(q,I)^{p/q}) \leq \frac{b^{-2\tau(\frac{p}{2})+\frac{p\tau(q)}{q}}}{1-b^{-\tau(p)+\frac{p\tau(q)}{q}}}\,.    
    \end{equation}
    If $2^k < p/q \leq 2^{k+1}$ for some $k \geq 1$, for any $1 \leq j \leq n$ and $I \in \Q_j$, then 
    \begin{equation}\label{eq:E_Y_2^k}
    \E(Y_{j,n}(q,I)^{p/q}) \leq \left({1-b^{-\tau(2q)+2\tau(q)}} \right)^{-2^{k+1}}.
    \end{equation}
\end{lemma}
\begin{proof}
   The proof is similar to the proof of Theorem 3 in \cite{Heur}. Instead of the arithmetic scales $k < p/q \leq k+1$ for $k \in \N$ used in \cite{Heur}, we proceed by updating $\E(Y_{j,n} (q,I)^{p/q})$ in dyadic scales, that is $2^k  < p/q \leq 2^{k+1}$. We provide the details for the reader's convenience.
    
    Note that
    $$
    Y_{j,n+1}(q,I) = b^{-d} \sum_{i \in \Lambda } W_{q,i} Y_{j+1,n+1} (q,I_i)
    $$
    where $I_i$, $i\in\Lambda$ are the elements of $\Q_{j+1}$ satisfying $I_j\subset I$ and $(W_{q,i})_{i\in\Lambda}$ is an independent copy of $W_q$. From now on, We simply write $\E(Y_{j,n+1,q} ) :=\E(Y_{j,n+1}(q,I))$ and $\E(Y_{j+1,n+1,q} ) :=\E(Y_{j+1,n+1}(q,I_i))$, since these expressions are independent of $I$ and $I_i$. Observe that
    \[
    \begin{split}
        Y_{j,n+1}(q,I)^{\frac{p}{q}} &\leq b^{-\frac{dp}{q}} \left[ \sum_{i \in \Lambda}W_{q,i}^{\frac{p}{2q}} Y_{j+1, n+1}(q,I_i)^{\frac{p}{2q}} \right]^{2}\\
        & = b^{-\frac{dp}{q}}\sum_{i \in \Lambda} W_{q,i}^{\frac{p}{q}} Y_{{j+1,n+1}}(q,I_{i})^{\frac{p}{q}} \\
        & \quad +b^{-\frac{dp}{q}}\sum_{i_1 \neq i_2 } W_{q,i_1}^{\frac{p}{2q}} Y_{{j+1,n+1}}(q,I_{i_1})^{\frac{p}{2q}} W_{q,i_2}^{\frac{p}{2q}} Y_{{j+1,n+1}}(q,I_{i_2})^{\frac{p}{2q}}.
    \end{split}
    \]
    Since $Y_{j+1,m}(q,I_i)^{p/q}$ is a submartingale in $m$,
    $$
    \E (Y_{j+1,n+1,q}^{p/q}) \leq \E( Y_{j+1,n+2,q}^{p/q}) = \E(Y_{j,n+1,q}^{p/q}). 
    $$
    Hence, we have
    \[
    \begin{split}
        \E(Y_{j,n+1,q}^{\frac{p}{q}}) & \leq b^{-\frac{dp}{q}}\sum_{i \in \Lambda}\E( W_{q,i}^{\frac{p}{q}})\E( Y_{{j,n+1,q}}^{\frac{p}{q}}) \\
        & \quad +b^{-\frac{dp}{q}}\E(Y_{{j+1,n+1,q}}^{\frac{p}{2q}})^2 \left[  \sum_{i\in\Lambda } \E(W_{q,i}^{\frac{p}{2q}}) \right]^2 .
    \end{split}
    \]
    Since 
    \begin{equation*}
    b^{-\frac{dp}{q}} \sum_{i\in \Lambda} \E(W_{q,i}^{p/q}) = b^{-\tau(p) + \frac{p\tau (q)}{q}}\,
    \end{equation*}
    and
    \begin{equation*}
    b^{-\frac{dp}{2q}} \sum_{i\in \Lambda} \E(W_{q,i}^{p/2q}) = b^{-\tau(\frac{p}{2}) + \frac{p\tau (q)}{2q}}\,,
    \end{equation*}
    we obtain that
    \begin{equation}\label{eq:E_Y_pq}
    \E (Y_{j,n+1,q}^{p/q}) \leq \frac{b^{-2\tau(\frac{p}{2})+\frac{p\tau(q)}{q}}}{1-b^{-\tau(p)+\frac{p\tau(q)}{q}}}\E(Y_{j+1,n+1,q}^{p/2q})^2.    
    \end{equation}

    Assume that $2^k < p/q \leq 2^{k+1}$ for some $k \geq 0$. If $k=0$, since $p/2q <1$, we have $\E(Y_{j+1,n+1,q}^{p/2q}) \leq 1$. Then, \eqref{eq:E_Y_pq} implies \eqref{eq:E_Y_1}. If $k \geq 1$, we first note that
      \begin{align*}
\E(Y_{j+1,n+1,q}^{p/2q})\le\E(Y_{j+1,n+1,q}^{2^k})\,,
    \end{align*}
    and $    b^{-2\tau(\frac{p}{2})+\frac{p\tau(q)}{q}} \leq 1$ for any $p,q$ such that $1 \leq q< p/2<q_{\max}$. Then, we consider $p_i = 2^iq $ for $1 \leq i \leq k$. Using \eqref{eq:E_Y_pq} inductively, we obtain that
    \[
    \E(Y_{j,n+1,q}^{p/q}) \leq \prod_{i=1}^{\min\{k,n-j\}} \left(1-b^{-\tau(2^{k+1-i}q)+2^{k+1-i}\tau(q)} \right)^{-2^i}\left( 1-b^{-\tau(p)+\frac{p\tau(q)}{q}}\right)^{-1}\,.
    \]
  
    Since $\tau(p)-\frac{p\tau(q)}{q}$ is an increasing function of $p < q_{\max}$ for a fixed $q$, we get
    $$
    b^{-\tau(2^{k+1-i}q)+2^{k+1-i}\tau(q)}  \geq b^{-\tau(2q) +2\tau(q)},
    $$
    and
    $$
    b^{-\tau(p)+\frac{p\tau(q)}{q}} \geq b^{-\tau(2q) +2\tau(q)}.
    $$
    Thus, we have established \eqref{eq:E_Y_2^k}.
\end{proof}

\begin{lemma}\label{lem:E(S(p,q,j,n))}
For $1 \leq p,q\leq \infty$ and $1\le j\le n$, we let
$$
S(p,q,j,n) = \left( \sum_{I \in \Q_j} \left( \sum_{J \in \Q_n , J \subset I} \nu_n(J)^q \right)^{p/q} \right)^{1/p}.
$$
Then, we have
\begin{equation}\label{eq:E(S(p,q,j,n))}
\E (S(p,q,j,n)) \lesssim b^{-j\widetilde{\tau}(p) / p  - (n-j) \widetilde{\tau}(q)/q}. 
\end{equation}
Note that $\lim_{p \rightarrow \infty} \widetilde{\tau}(p)/p =\alpha_{\min}$. Thus, in \eqref{eq:E(S(p,q,j,n))}, we adopt the convention $\widetilde{\tau}(\infty) / \infty =\alpha_{\min}  $ if $p=\infty$ or $q=\infty$.
\end{lemma}
\begin{proof}
    For any $1\le j\le n$ and $I \in \Q_j$, we denote 
    $$S(q,I,n) = \left( \sum_{J \in \Q_n, J \subset I} \nu_n(J)^q \right)^{1/q}, $$
        so that $S(p,q,j,n) = (\sum_{I \in \Q_j} S(q,I,n)^p)^{1/p} $.
        
    First, we consider the case when $p, q \leq q_{\max} < \infty$. If $p \leq q$, the required estimate can be easily derived by the Minkowski inequality: Since $p/q \leq 1$, conditional on $\nu_j$,  we have
        $$
        \E \left(  S(q,I,n)^p | \nu_j \right) \leq \E \left( \sum_{J \in \Q_n, J \subset I}\nu_n(J)^q|\nu_j \right)^{p/q} =  b^{-(n-j) p\tau(q)/q} \nu_j (I)^p.
        $$
    uniformly in $\nu_j$.
    Therefore,
    $$
    \E(S(p,q,j,n)|\nu_j)  \leq  \left( \sum_{I \in \Q_j}  \E(S(q,I,n)^p | \nu_j) \right)^{1/p} =  b^{-(n-j)\tau(q)/q}\left(\sum_{I \in \Q_j} \nu_j(I)^p \right)^{1/p}.
    $$
    By the law of total probability, we thus obtain
    $$
    \E(S(p,q,j,n)) \leq  b^{-(n-j)\tau(q)/q}\,\E\left(\sum_{I \in \Q_j} \nu_j(I)^p \right)^{1/p} = b^{-(n-j) \tau(q)/q -j\tau(p)/p}.
    $$
    If $p, q \leq q_{\max}$, recall that $\widetilde{\tau}(p) = \tau(p)$ and $\widetilde{\tau}(q)= \tau(q)$, respectively. Thus, we obtain \eqref{eq:E(S(p,q,j,n))}.

    Next, let us consider the case $p,q \leq q_{\max} < \infty$ and $p > q$. We have
    \begin{equation}\label{eq:S_to_Y}
    \begin{split}
        S(q,I,n) &=  b^{-(n-j)d}\left(\sum_{i\in\Lambda}\E(W_i^q)\right)^{(n-j)/q} \nu_j(I) Y_{j,n}(q,I)^{1/q}\\
        &= b^{-(n-j)\tau(q)/q}\nu_j(I)  Y_{j,n}(q,I)^{1/q}.
    \end{split}
    \end{equation}
    Therefore, for each  fixed $\nu_j$, we have 
    $$
        \E \left(  S(q,I,n)^p | \nu_j \right) \leq   b^{-(n-j) p \tau(q) /q} \nu_j(I)^p \E \left( Y_{j,n} (q,I)^{p/q}|\nu_j \right).
    $$
    By Lemma \ref{lem:E_Y_jn}, we obtain that $\E (Y_{j,n} (q,I)^{p/q}|\nu_j ) < \infty$ uniformly in $j$, $n$, and $I$. Since $Y_{j,n}(q,I)$ and $\nu_j$ are independent, we obtain that
    $$
        \E \left(  S(q,I,n)^p | \nu_j \right) \lesssim b^{-(n-j) p\tau(q) /q}  \nu_j(I)^p 
    $$
    uniformly in $\nu_j$. Repeating the remaining steps as in the case $p,q \leq q_{\max}$ and $p \leq q$, we obtain the desired estimate \eqref{eq:E(S(p,q,j,n))}.
    
    If $p > q_{\max}$ or $q > q_{\max}$, we get
    $$
    S(p,q, j,n) \leq S(\min(p,q_{\max}), \min(q, q_{\max}), j, n).
    $$
    Since $\widetilde{\tau}(r)/ r = \tau(q_{\max}) / q_{\max}$ for all $r \geq q_{\max}$, we obtain \eqref{eq:E(S(p,q,j,n))}.

    Lastly, let us consider the case $q_{\max} =\infty$. It suffices to consider the case when $p = \infty$ or $q=\infty$. Otherwise, we can repeat the argument above. If $q  < q_{\max}$ and $n,j$ are fixed, then $\widetilde{\tau}(p)/p \longrightarrow \alpha_{\min}$ and $ S(p,q,j,n)  \longrightarrow S(\infty, q, j,n) $ as $p \rightarrow \infty$. Also, Lemma \ref{lem:E_Y_jn} implies that $\E(Y_{j,n,q}^{p/q}) \leq C_q^p$ when $p>2q$, where 
    $$C_q = (1-b^{-\tau(2q)+2\tau(q)})^{-2/q}.$$
    Therefore, noting \eqref{eq:S_to_Y}, the implicit constant in \eqref{eq:E(S(p,q,j,n))} is uniform in $p$ as $p\longrightarrow\infty$. Thus, \eqref{eq:E(S(p,q,j,n))} for $p=q_{\max}=\infty$ easily follows by Fatou's lemma. Next, recall that the implicit constant in \eqref{eq:E(S(p,q,j,n))} equals $1$ when $p \leq q$ and thus \eqref{eq:E(S(p,q,j,n))} for $p<\infty$ and $q=q_{\max}=\infty$ follows by the same reasoning as above.  If $q_{\max}=p=q=\infty$, we use that $S(\infty, \infty, j,n) = S(\infty, 1, n,n)$ and \eqref{eq:E(S(p,q,j,n))} easily follows. 
\end{proof}

For $1 \leq p, q \leq \infty$, we define
$$
\varepsilon_{p,q,n} = \frac{1}{n} \sup\left\{ \log(S(p,q,j,n) ) + j \widetilde{\tau}(p)/p + (n-j) \widetilde{\tau}(q)/q  : 0 \leq j \leq n \right\}.
$$

\begin{lemma}\label{lem:unif_pq}
    For fixed $1 \leq p,q \leq \infty$, $\lim_{n \rightarrow \infty} \varepsilon_{p,q,n} =0$ almost surely on non-extinction
\end{lemma}
\begin{proof}

    For any $\varepsilon >0$,  Lemma \ref{lem:E(S(p,q,j,n))} implies that
    $$
        \PP \left( S(p,q,j,n) > b^{-j\widetilde{\tau}(p)/p - (n-j)\widetilde{\tau}(q)/q + n\varepsilon}  \right) \leq b^{-n\varepsilon},
    $$
    and we obtain
    $$ \sum_{n \geq 0}\sum_{0 \leq j \leq n} \PP \left( S(p,q,j,n) > b^{-j\widetilde{\tau}(p)/p - (n-j)\widetilde{\tau}(q)/q + n\varepsilon}  \right) <  \sum_{n \geq 0} \sum_{ 0 \leq j \leq n }b^{-n \varepsilon} <\infty.$$
    Borel-Cantelli lemma implies that, almost surely, there are only finitely many pairs $0 \leq j \leq n$ such that 
    $$
    S(p,q,j,n) > b^{-j\widetilde{\tau}(p)/p - (n-j)\widetilde{\tau}(q)/q + n\varepsilon}. 
    $$
In particular, almost surely, there is $n_0\in\N$ such that
\[\sup\left\{ \log(S(p,q,j,n) ) + j \widetilde{\tau}(p)/p + (n-j) \widetilde{\tau}(q)/q  : 0 \leq j \leq n \right\}\le\varepsilon n\]
for all $n\ge n_0$.    
\end{proof}

\begin{remark}
In the proof of the main theorems, we only use the cases $q=1$ and $q=2$, but the Lemmas \ref{lem:E(S(p,q,j,n))} and \ref{lem:unif_pq} work for all $1 \leq p, q \leq \infty$. 
\end{remark}

\subsection{A concentration inequality}\label{sec:conc}
We complete this section with our key concentration inequality, a variant of \cite[Proposition C.3]{AK}. 

\begin{lemma}\label{lem:conc}
For some $p >4$, let $\Phi(t)\lesssim t^{-p}$ when $t \geq 1$. Let $X_1,\ldots, X_N$ be independent random variables with zero expectation such that
\begin{equation}\label{eq:tail_bd}
\PP(X_k>t)\le\Phi(t)    
\end{equation}
for all $1\le k\le N$. Then, for all $a_1,a_2,\ldots a_N\ge 0$, $M>1$, and $t>0$,
\[\PP\left(\sum_{k=1}^N a_k X_k>t\right)=N\Phi(M)+\exp\left(-\lambda t+O(\lambda^2)\sum_{k=1}^N a_k^2\right)\,,\]
where 
\begin{equation}\label{eq:lambda}
    \lambda=\frac{q\log M}{M\max_{1\le k\le N}a_k}\,.
\end{equation}
and $q\leq p/2-1$. The $O$-constant only depends on $p$ and $\Phi$. 
\end{lemma}

\begin{proof}
Denote $X=\sum_{k=1}^N a_k X_k$,
$U_k=\min\{M,X_k\}$,
$Y_k=a_k U_k$, and $Y=\sum_{k=1}^N Y_k$. Clearly,
\begin{align}
\PP\left(X>t\right)\le\PP(Y>t)+\PP\left(\max_{1\le k\le N}X_k>M\right)\,.
\end{align}
The last term is estimated by the union bound
\begin{equation}\label{eq:union_b}
   \PP\left(\max_{1\le k\le N}X_k>M\right)\le N\Phi(M)\,.
\end{equation}

We proceed to estimate $\PP(Y>t)$ using the exponential moment method. Let $\lambda>0$ be a constant to be determined later. Using Markov's inequality, we have
\begin{equation}\label{eq:exp_moment}
\begin{split}
    \PP(Y>t)&=\PP(\exp(\lambda Y)>\exp(\lambda t))\\
    &\le\exp(-\lambda t)\E\left(\exp(\lambda Y\right))\\
    &=\exp(-\lambda t)\prod_{k=1}^N\E\left(\exp(\lambda a_k U_k)\right)\,.
    \end{split}
\end{equation}
For each $a>0$, we have
\begin{equation}
    \E(\exp(a U_k)))\le1+\frac{a^2}{2}\E\left(U_k^2\exp(a\max\{U_k,0\})\right)\,,
\end{equation}
see \cite[(C.27)]{AK}. 
Let $C=C(\Phi,p)$ such that
\begin{equation}\label{eq:tail_log}
\log\Phi(t)\le-p\log t+C\,.
\end{equation}
Splitting the integral into three parts, using change of variables and the tail bound \eqref{eq:tail_bd} and \eqref{eq:tail_log} yields
\begin{equation}\label{eq:int}
\begin{split}
    &\E\left(U_k^2\exp(a\max\{U_k,0\})\right)\\
    &=\int_{U_k<0}U_k^2\,d\PP+\int_{0\le U_k\le 1}U_k^2\exp(a U_k)\,d\PP+\int_{1<U_k\le M}U_k^2\exp(a U_k)\,d\PP\\
    &=\int_{X_k<0}X_k^2\,d\PP+\int_{0\le X_k\le 1}X_k^2\exp(a X_k)\,d\PP+\int_{\exp(a)}^{M^2 \exp(aM)}\PP\left(X_k^2\exp(a X_k)>t\right)\,dt\\
    &\le\E(X_k^2)+\exp(a)+\int_{s=1}^M(2s+as^2)\exp(a s)\PP\left(X_k>s\right)\,ds\\
    &\le\E(X_k^2)+\exp(a)+\int_{s=1}^M(2s+a s^2)\exp\left(as+C-p\log s\right)\,ds\,.
\end{split}
\end{equation}
If $a\leq \tfrac{q\log M}M$, then (since $\log$ is convex)
\[as\le a+q\log s\,,\]
for all $1\le s\le M$. Whence
\begin{align*}
    \exp\left(C+as-p\log s\right)\le\exp\left(C+a-(1+p/2)\log s\right)\lesssim s^{-1-p/2}\,.
\end{align*}
The implicit constant only depends on $p$ and $\Phi$, since $a \leq \frac{q \log M}{M} \lesssim q \lesssim p$ and $C$ depends on $\Phi$ and $p$. Combining with \eqref{eq:int} and noting that $\E(X_k^2)\lesssim 1$ by \eqref{eq:tail_bd} implies that 
\[\E\left(U_k^2\exp(a\max\{U_k,0\})\right)\lesssim 1\,,\] 
for all $1\le k\le N$. Noting \eqref{eq:lambda} and combining with \eqref{eq:exp_moment}, we have
\begin{align*}
    \PP(Y>t)\le\exp\left(-\lambda t+O(\lambda^2)\sum_{k=1}^N a_k^2\right)\,.
\end{align*}
Combining with \eqref{eq:union_b}, this gives the claim.
\end{proof}

\section{The lower bound of the Fourier dimension}\label{sec:Other Lemmas}

We now turn to the main novel feature in this work, the exact value of the Fourier dimension for curvilinear cascades. First, we prove the lower bound of the Fourier dimension.

\begin{theorem}\label{FD_up}
    Almost surely, $|\widehat{\mu}(\xi)|\lesssim_\beta |\xi|^{-\beta}$, if $\beta<\alpha_{\min}/2$.
\end{theorem}
This theorem implies that $\dim_F(\mu) \geq \alpha_{\min}$. We will consider different estimates according to the size of $|\xi|$.

\begin{lemma}\label{lem:small_f}
If $\beta<\widetilde{\tau}(2)$, then 
    $|\widehat{\mu_{n+1}}(\xi)-\widehat{\mu_n}(\xi)|\lesssim_\beta b^{-n\beta/2}$ for all $n$ and all $|\xi|\le b^{n}$, where the implicit constant is (random and) independent of $n$, $\xi$.
\end{lemma}
We defer the proof of Lemma \ref{lem:small_f} to the appendix, see Lemma \ref{main_probab_lemma} (and Remark \ref{rem:appendix}), where a slightly more general version is obtained. We note that for Theorem \ref{FD_up}, we only need the bound for $\beta<\alpha_{\min}$. The full power of the lemma will be used for the spherical averages in Section \ref{sec:spherical}. Note also that the curvature assumption  $\det(\gamma'(t), \gamma''(t)) \neq 0$ is not used in the proof of Lemma \ref{lem:small_f}.

We will use the following, which is an immediate consequence of our curvature assumption and the Van der Corput lemma (see e.g. \cite[Theorem 14.2]{Mattila2015}).

\begin{lemma}\label{lem:Vander}
    Let $I\subset[0,1]$. Then
    \begin{equation*}
    \left|\int_{\gamma(I)}\exp\left(-2\pi i x\cdot\xi \right)\,dx\right|\lesssim |\xi|^{-1/2}\,.
\end{equation*}
    For $j \geq 1 $, assume that $b^{-j}|\xi| \leq |\gamma'(t) \cdot \xi|$ for all $t \in I$. Then
    \begin{equation*}
    \label{eq:VC2}\left|\int_{\gamma(I)}\exp\left(-2\pi i x\cdot\xi \right)\,dx\right|\lesssim b^j/|\xi|\,,    
    \end{equation*}
\end{lemma}

We will use the following notation in the proofs of Lemmas \ref{lem:large_f} -- \ref{lem:intrem_f}, and \ref{lem:circ_decay_large}. Given $\xi$, let $\mathcal{I}_1=\mathcal{I}_1(\xi)$ consist of intervals $I\in \D_n$ such that $b^{-1} |\xi|  \leq \min\{ |\gamma'(t) \cdot \xi|: \gamma(t) \in I \}$. For $2 \leq j < n$, let $\mathcal{I}_j=\mathcal{I}_j(\xi)$ consist of those intervals  $I \in \D_n$ such that $b^{-j} |\xi| \leq \min\{ |\gamma'(t) \cdot \xi|: \gamma(t) \in I \} < b^{-j+1}  |\xi| $ and let $\mathcal{I}_n=\mathcal{I}_n(\xi)=\D_n\setminus\cup_{j=1}^{n-1}\mathcal{I}_j$. Note that $\cup_{I \in \mathcal{I}_j}I$ is contained in a union of $O(1)$-many arcs of length $\sim b^{-j}$. Let us denote the union of these arcs by $\mathcal{S}_j=\mathcal{S}_j(\xi)$.

Now, we can prove the estimates for $\widehat{\mu_n}(\xi)$ in different scales of $\xi$. Throughout this section, let us denote $\varepsilon_n = \varepsilon_{\infty,1,n }$ for notational convenience.

\begin{lemma}\label{lem:large_f}
    If $ \beta < \alpha_{\min}$, then $\widehat{\mu_n}(\xi)\lesssim_\beta |\xi|^{(\varepsilon_n-\beta)/2}$ for all $|\xi|\ge b^{2n}$, where the implicit constant is (deterministic and) independent of $n$ and $\xi$. 
\end{lemma}
\begin{proof}
    Consider $\xi\in\R^2$, $|\xi|\ge b^{2n}$. Using Lemma \ref{lem:Vander} for each $I\in\D_n$ and summing over all intervals $I\in\mathcal{I}_j$ yields that if $1 \leq j <n$, then
\begin{align}\label{eq:large_f}
\left|\sum_{I\in\mathcal{I}_j}\int_I\exp\left(-2\pi i x\cdot\xi \right)\,d\mu_n(x)\right|&\lesssim b^{n+j}\mu_n( \mathcal{S}_j)/|\xi| \nonumber \\
&\lesssim b^{n+j(1-\alpha_{\min})+n\varepsilon_n}/|\xi|\,,
\end{align}
and 
\begin{align*}
\left|\sum_{I\in\mathcal{I}_n}\int_I\exp\left(-2\pi i x\cdot\xi \right)\,d\mu_n(x)\right|& \lesssim b^{n}\mu_n( \mathcal{S}_n)|\xi|^{-1/2}\\
& \lesssim b^{n(1-\alpha_{\min}+\varepsilon_n)}|\xi|^{-1/2}\,.
\end{align*}

The last inequality in each estimate follows from the definition of $\varepsilon_{p,q,n}$ and the facts that $\widetilde{\tau}(\infty)/\infty = \alpha_{\min}$ and $\widetilde{\tau}(1)=0$. Note that the sum of \eqref{eq:large_f} over $1 \leq j \leq n$ is $ \lesssim nb^{2-\alpha_{\min} +\varepsilon_n} |\xi|^{-1}
$ and we can drop $n$ if $ \alpha_{\min} < 1$. This yields
\[|\widehat{\mu_n}(\xi)|\lesssim\left(b^{n(1-\alpha_{\min}+\varepsilon_n)}|\xi|^{-1/2}+ n b^{n(2-\alpha_{\min}+\varepsilon_n)}|\xi|^{-1}\right) \lesssim_\beta |\xi|^{(\varepsilon_n-\beta)/2}\,,\]
proving the lemma. 
\end{proof}
In the intermediate scales, we combine the concentration inequality (Lemma \ref{lem:conc}) with van der Corput's lemma. Some elementary but technical parts of the proof are identical to those in the proof of Lemma \ref{main_probab_lemma}, and they are omitted here.

 For $t>0$, we let $A_t$ denote the union of all $I \in \D_n$ which intersect the closed $t$-neighbourhood of a set $A\subset\R^d$.
 
\begin{lemma}\label{lem:conc_app}
For any $\delta>0$, almost surely,
\[\left|\widehat{\mu_{n+1}}(\xi)-\widehat{\mu_n}(\xi)\right|\lesssim_\delta 
b^{2n \delta} \left( \sum_{1 \leq \ell \leq k} b^{2(\ell-k)} \sum_{\iii\in\Lambda^n,D_\iii\subseteq (\mathcal{S}_\ell(\xi))_{b^{1-n}}} \mu_n(D_\iii)^2 \right)^{1/2}+b^{-n}\,,\]
for all $1\le k\le n$ and $b^{n+k-1} \leq |\xi| \leq b^{n+k}$. Here, the implicit constant is random, but independent of $n,k$ and $\xi$.

\end{lemma}

\begin{proof}
    For $1 \leq \ell \leq k$ and for each $D_\iii \in \mathcal{I}_\ell$, let 
    $$
    a_{\iii,\xi} = b^{\ell-k} \mu_n(D_\iii)
    $$
    and
    $$
    X_{\iii,\xi} =  b^{n+k-\ell} \sum_{j \in \Lambda} ((W_\iii)_j-1) \int_{D_{\iii,j}} \exp(-2\pi i x \cdot \xi) dx
    $$
    where  $(W_\iii)_j$ denotes an independent copy of $W$ and $D_{\iii,j} = \gamma(Q_{\iii,j})$ where $Q_{\iii,j} = x_{(\iii,j)} + [0,b^{-(n+1)})$ with $(\iii,j) = (i_1,i_2, \cdots, i_n, j)$ for each $\iii \in \Lambda^n$ and $j \in \Lambda$. Then, we can write
    $$
    \widehat{\mu_{n+1}} (\xi) - \widehat{\mu_n}(\xi)  = \sum_{\iii \in \Lambda^n}  a_{\iii,\xi} X_{\iii, \xi}.
    $$
    Lemma \ref{lem:Vander} implies that
    $$
     b^{n+k-\ell} \int_{D_{\iii,j}} \exp(-2\pi i x \cdot \xi) dx \lesssim 1
    $$
    uniformly in $D_{\iii,j}, n,k$ and $\ell$. Also, \eqref{eq:moments} implies that
    \begin{equation*} 
    \PP\left(\max_{j\in\Lambda} W_j>t+1\right)\lesssim_p t^{-p}
    \end{equation*}
    for all $0<p<\infty$. Thus,
    $$
    \PP(X_{\iii,\xi} >t) \lesssim_p t^{-p}
    $$
    for all $X_{\iii,\xi}$ when $t >1$. For each $n$ and $k$, denote
    $$
    S_{n,k,\xi} = \sum_{0 \leq \ell \leq k} b^{2(\ell -k)} \sum_{D_\iii \subseteq \mathcal{S}_\ell(\xi)} \mu_n(D_\iii)^2.
    $$
    We may use Lemma \ref{lem:conc} with $N=b^{n+1}$, $q=1$, $t = b^{2n\delta} S_{n,k,\xi}^{1/2}$ and $\lambda = t^{-1}b^{n\delta}$ and conclude (see the proof of Lemma \ref{main_probab_lemma}) that 
    $$
    \PP(|\widehat{\mu_{n+1}}(\xi) - \widehat{\mu_n}(\xi)| \geq b^{2n\delta}S_{n,k,\xi}^{1/2} ) \leq 2b^{n+1} \Phi(M) + 2\exp (-b^{n\delta }  + O(b^{-2n\delta}))
    $$
    where
    $$\frac{\log(M)}{M } = \lambda \max\{a_{\iii,\xi}\,:\,\iii\in\Lambda_n\} \lesssim b^{-n\delta}. $$
    Choosing $p > 7/\delta$, this implies that
\begin{align*}
    \PP\left(|\widehat{\mu_{n+1}}(\xi)-\widehat{\mu_n}(\xi)|\ge b^{2n\delta}S_{n,k,\xi}^{1/2} \text{ for some }\xi\in b^{-n}\mathbb{Z}^2\,,b^{n} \le|\xi|\le b^{2n}\right) \lesssim_\delta b^{-cn} 
\end{align*}
for some $c>0$.
    The Borel-Cantelli lemma implies that almost surely 
    \begin{equation} \label{eq:mu_K}
     |\widehat{\mu_{n+1}}(\xi) - \widehat{\mu_n}(\xi)| \lesssim_\delta b^{2n\delta}S_{n,k,\xi}^{1/2},   
    \end{equation}
    for all $\xi \in b^{-n} \Z^2$. Thus, for any $b^{n+k-1} \leq |\xi| \leq b^{n+k}$, there exists $\xi' \in b^{-n}\Z$ such that $|\xi -\xi'|_\infty \leq b^{-n}$ and \eqref{eq:mu_K} holds for $\xi'$.

    For $\Gamma=\gamma([0,1])$, $\mu_n(\Gamma)$ is a martingale with $\E(\mu_n(\Gamma)) =1$. Hence, $\sup_n \mu_n ( \Gamma)< \infty$ almost surely. This implies that $\widehat{\mu_n}(\xi)$ is Lipschitz with a Lipschitz constant independent of $n$ and using this, \eqref{eq:mu_K}, and the fact $\mathcal{S}_\ell(\xi')\subset(\mathcal{S}_\ell(\xi))_{b^{1-n}}$, if $|\xi-\xi'|_\infty\le b^{-n}$, we obtain the desired conclusion for all $b^{n+k-1} \leq |\xi| \leq b^{n+k}$. 
\end{proof}

\begin{lemma}\label{lem:intrem_f}
    Let $\beta<\alpha_{\min}$. Then, almost surely,
\[|\widehat{\mu_{n+1}}(\xi)-\widehat{\mu_n}(\xi)| \lesssim_\beta |\xi|^{-\beta/2}\]
for all $n$ and all $b^n\le|\xi|\le b^{2n}$, where the implicit constant is random, but independent of $n$ and $\xi$.
\end{lemma}

\begin{proof}
Let $0<\delta<(\alpha_{\min}-\beta)/4$, $1\le k\le n$ and $b^{n+k-1}\le|\xi|\le b^{n+k}$. Recall that $\mathcal{S}_\ell=\mathcal{S}_\ell(\xi)$ is a union of $O(1)$-many arcs of length $\sim b^{-\ell}$. Thus, we have
\[
\begin{split}
\sum_{\iii\in\Lambda^n,D_\iii \subseteq (\mathcal{S}_\ell(\xi))_{b^{1-n}}} \mu_n(D_\iii)^2 &\leq  b^{-n\alpha_{\min} +n \varepsilon_n} \sum_{\iii\in\Lambda^n,D_\iii \subseteq (\mathcal{S}_\ell(\xi))_{b^{1-n}}} \mu_n(D_\iii) \\
&\lesssim b^{-(n+\ell) \alpha_{\min} + 2n\varepsilon_n} \,.
\end{split}
\]
Therefore, 
\begin{equation}\label{eq:S_ell_sum}
    \sum_{0 \leq \ell \leq k} b^{2(\ell -k)} \sum_{\iii\in\Lambda^n,D_\iii \subseteq (\mathcal{S}_\ell(\xi))_{b^{1-n}}} \mu_n(D_\iii)^2\lesssim b^{2n\varepsilon_n-(n+k)\alpha_{\min}}\sim b^{2n\varepsilon_n}|\xi|^{-\alpha_{\min}}\,.
\end{equation}
Since $|\xi|\leq b^{2n}$ and $\alpha_{\min} \leq 1$,  we also have $b^{-n} \lesssim |\xi|^{-\alpha_{\min}/2}$. Thus, combining \eqref{eq:S_ell_sum} and Lemma \ref{lem:conc_app}, it follows that
\begin{align*}
    |\widehat{\mu_{n+1}}(\xi)-\widehat{\mu_n}(\xi)|
   \lesssim |\xi|^{-\alpha_{\min}/2 }b^{n(2\delta+ \varepsilon_n)}.
\end{align*}
This implies the claim since Lemma \ref{lem:unif_pq} applied with  $p =\infty$ and $q=1$ yields $\varepsilon_n\longrightarrow 0$ as $n\to\infty$, almost surely.
\end{proof}

Now, we are ready to prove Theorem \ref{FD_up}.
\begin{proof}[Proof of Theorem \ref{FD_up}]
    Let $0<\beta<\beta'<\alpha_{\min}$. Using Lemmas \ref{lem:small_f} and \ref{lem:intrem_f}, 
    it follows that for some random constant $K>0$,
    $|\widehat{\mu_{n+1}}(\xi)-\widehat{\mu_n}(\xi)|\le K b^{-\beta n/2}$ for all $n\in\N$ and all $|\xi|\le b^n$ and, moreover, that
$|\widehat{\mu_{n+1}}(\xi)-\widehat{\mu_n}(\xi)|\le K |\xi|^{-\beta'/2}$ if $b^n\le|\xi|\le b^{2n}$.

Consider $\xi$ with $b^{2j}\le |\xi|\le b^{2j+2}$, $j\in\mathbb{N}$ and let $m\ge 2j$. Applying Lemma \ref{lem:large_f}, we get
\begin{align*}
    &|\widehat{\mu_j}(\xi)|+\sum_{n=j}^m|\widehat{\mu_{n+1}}(\xi)-\widehat{\mu_n}(\xi)|\\
    &=|\widehat{\mu_j}(\xi)|+\sum_{n=j}^{2j}|\widehat{\mu_{n+1}}(\xi)-\widehat{\mu_n}(\xi)|+\sum_{n=2j}^m|\widehat{\mu_{n+1}}(\xi)-\widehat{\mu_n}(\xi)|\\
    &\lesssim_\beta |\xi|^{(\varepsilon_j-\beta)/2}+|\xi|^{-\beta/2}+|\xi|^{-\beta/2}\,,
\end{align*}
where we have used that $j\lesssim\log|\xi|\lesssim|\xi|^{(\beta'-\beta)/2}$.
Noting that the implicit constant is independent of $\xi$, that $j\to\infty$ as $|\xi|\to\infty$, and that $\varepsilon_j\to 0$ as $j\to\infty$ (by Lemma \ref{lem:unif_pq} applied with $p=\infty$ and $q=1$), we get
$|\widehat{\mu}(\xi)|=\lim_{m\to\infty}|\widehat{\mu_m}(\xi)|\lesssim_\beta |\xi|^{-\beta/2}$.
\end{proof}

\section{The upper bound of the Fourier dimension}\label{sec:ub}
The upper bound on the Fourier dimension, $\dim_F\mu\le \alpha_{\min}$, is a general fact valid for all measures supported on sufficiently regular curves. The following lemma is likely known, but we provide a proof since we have not found the result in the literature. For each $\theta \in S^{d-1}$, we let $P_\theta: \R^d \rightarrow \R$ denote the projection
\[
P_\theta(x) = x \cdot \theta\,,
\]
and let $\eta_\theta$ be the image measure defined as $\eta_\theta(B) =\eta(P_\theta^{-1} (B))$ for $B\subset\R$.

\begin{lemma}\label{lem:FD_ub}
    Let $\eta$ be a finite Borel measure supported on a $C^2$-curve $\Gamma\subset\R^2$ defined by $\gamma(t): [0,1] \rightarrow \mathbb{R}^2$ with $\det(\gamma'(t),\gamma''(t)) \neq 0$. For all $x\in\spt\eta$,
    \[\dim_F\eta\le\dim(\eta,x)\,.\]     
\end{lemma}
The estimate $\dim_F \mu \leq \alpha_{\min}$ is an immediate corollary to Lemma \ref{lem:FD_ub}.
\begin{proof}
    Let $x'\in\spt\eta$. We will show that $\dim_2 \eta_\theta\le\dim(\eta,x')$, where $\theta\in S^{1}$ is the unit normal of $\Gamma$ at $x'$. Since $\dim_F \eta \leq \dim_H \eta \leq  1$, the identity
    \[
\widehat{\eta_\theta}( \xi) = \widehat{\eta} ( \xi \theta)\,,
\]
    implies that $\dim_F\eta\le\dim_F\eta_\theta $. Then, it follows that $$\dim_F\eta\le\dim_F\eta_\theta\le\dim_2\eta_\theta\le\dim(\eta,x').$$

    Let $t>s>\dim(\eta,x')$. It suffices to show that $\dim_2\eta_\theta\le t$. To that end, we recall that
    \[\dim_2\eta_\theta=\sup\left\{0\le h< 1\,:\,\int\int  |x-y|^{-h} d\eta_\theta(x)d\eta_\theta(y)<\infty\right\}\,,\]
    see e.g. \cite[Proposition 2.1]{HuntKaloshin}.
    Since $\det(\gamma'(t),\gamma''(t))>0$ and $\theta$ is the unit normal of $\Gamma$ at $x'$, there is $1 \leq C<\infty$, such that $\Gamma\cap B(x',r^{1/2})\subset P_\theta^{-1}(B(x,Cr))$ for all $0<r<1$, and for all $x\in B(P_\theta(x'),r))\subset\R$. Now
 \[
    \begin{split}
        \iint |x-y|^{-t} d\eta_\theta(x) d\eta_\theta(y)
        &\gtrsim r^{-t} \int \eta_\theta(B(x,Cr)) d\eta_\theta(x)\,,\\
        &\ge r^{-t}\int_{B(P_\theta(x'),r)} \eta(B(x',r^{1/2})) d\eta_\theta(x)\\
        &\gtrsim r^{-t}\eta\left(B(x',C^{-1/2}r^{1/2})\right)^2
    \end{split}
    \]
    for all $0<r<1$. Now, there are arbitrarily small values $0<r<1$, such that $\eta(B(x',C^{-1/2}r^{1/2}))>r^{s/2}$, and for these values of $r$, we thus have
    \begin{align*}
        \iint |x-y|^{-t} d\eta_\theta(x) d\eta_\theta(y) &\gtrsim r^{s-t}\,,
    \end{align*}
implying that $\dim_2\eta_\theta\le t$.
\end{proof}

\section{Decay of the spherical average}\label{sec:spherical}
Recall the definition of the spherical average $\sigma_p$ from \eqref{def:sp}.
\begin{theorem}\label{thm:circ_decay}
    Let $\beta <  \min \{ \widetilde{\tau}(2) , (1+\widetilde{\tau}(p))/p \} $. For $1\leq p \leq\infty$, almost surely we have
    $$
    \sigma_p(\mu)(r) \lesssim_\beta r^{-\beta/2}.
    $$
\end{theorem}
\begin{remark}
\emph{a)} If $1 \leq p \leq 2$, then 
\begin{equation}\label{eq:min_a_2}
    \min \{ \widetilde{\tau}(2) , (1+\widetilde{\tau}(p))/p \}=\widetilde{\tau}(2)\,.
\end{equation} 
Indeed, since $(1+\widetilde{\tau}(p)) /p$ is a slope between $(p,\widetilde{\tau}(p))$ and $(0,-1)$, and $\widetilde{\tau}(p)$ is a concave function of $p$, $(1+\widetilde{\tau}(2))/2 \leq (1+\widetilde{\tau}(p))/ p $ if $ 1 \leq p \leq 2$. Thus, \eqref{eq:min_a_2} follows from the fact that $\widetilde{\tau}(2) \leq 1$.\\
\emph{b)} The borderline case $p=\infty$ is equivalent to Theorem \ref{FD_up} and thus we assume in the proof that $p<\infty$. We note that with natural $L^\infty$-interpretations, the proof below would also cover the $p=\infty$ case (and this would essentially repeat the proof of Theorem \ref{FD_up}). 
\end{remark}

\begin{lemma}\label{lem:circ_decay_large}
    If $\beta < (1+\widetilde{\tau}(p))/p$, then
    $
    \sigma_p(\mu_n)(r) \lesssim_\beta r^{(\varepsilon_{p,1,n} - \beta)/2},
    $ for all $r \geq b^{2n}$, where the implicit constant is (deterministic and) independent of $n$.
\end{lemma}
\begin{proof}
    Recall the notations $\mathcal{I}_j(\xi)$, $\mathcal{S}_j (\xi)$ from Section \ref{sec:Other Lemmas}. If $\theta\in S^1$ and  $I \in \mathcal{I}_j(\theta)$ for $1 \leq j \leq n-1$, we use Lemma \ref{lem:Vander} and obtain that
    $$
    \left|\int_I \exp(-2\pi i r \theta\cdot x) d\mu_n(x)\right| \lesssim b^{n+j} \mu_n(I) r^{-1} ,
    $$
    and if $ I \in \mathcal{I}_n(\theta)$, then 
    $$
    \left|\int_I \exp(-2\pi i r \theta\cdot x) d\mu_n(x)\right| \lesssim b^n \mu_n(I) r^{-1/2} .
    $$
Therefore, if $1 \leq j \leq n-1$, then we have
\[
        \norm{\sum_{I \in \mathcal{I}_j(\theta)} \int_I \exp(-2\pi i r\theta\cdot x) d\mu_n(x)}_{L^p(d\sigma)} \lesssim b^{n+j} r^{-1} \norm{\mu_n (\mathcal{S}_j(\theta))}_{L^p(d\sigma)}\,.
\]

Recall that $\mathcal{S}_j(\theta)$ is contained in a union of $O(1)$-many arcs of length $b^{-j}$, and that $S^1(I) := \{\theta\in S^1\mid I \cap \mathcal{S}_j(\theta) \neq \varnothing \}$ is an arc of length $\sim b^{-j}$ for $I \in \D_j$. 
Thus, we obtain
\begin{align*}
    \norm{\mu_n (\mathcal{S}_j(\theta))}_{L^p(d\sigma)}^p  &\lesssim \int \left|\sum_{I \in \D_j, I \cap \mathcal{S}_j(\theta) \neq \varnothing }\mu_n (I)\right|^p d\sigma(\theta) \\
    &\lesssim \sum_{I \in \D_j} \mu_n(I)^p  \sigma ( \mathcal{S}^1(I) )\\
    &\lesssim b^{-j} \sum_{I \in \D_j} \mu_n(I)^p.
\end{align*}
Putting things together, we have
    \begin{align*}
        \norm{\sum_{I \in \mathcal{I}_j(\theta)} \int_I \exp(-2\pi i r\theta\cdot x) d\mu_n(x)}_{L^p(d\sigma)} 
        &\lesssim b^{n+j} r^{-1} \left(b^{-j} \sum_{I \in \D_j} \mu_n(I)^p \right)^{1/p}\\
        &\lesssim (b^{n(1+\varepsilon_{p,1,n})+j(1-1/p - \widetilde{\tau}(p)/p)})r^{-1}\,,
    \end{align*}
    where we used the fact that $\widetilde{\tau}(1)=0$. Similarly, if $n=j$, then  
    \begin{align*}
        \norm{\sum_{I \in \mathcal{I}_n(\theta)} \int_I \exp(-2\pi i r \theta\cdot x) d\mu_n(x)}_{L^p(d\sigma)} &\lesssim  b^{n} r^{-1/2} \norm{\mu_n (\mathcal{S}_j(\theta))}_{L^p(d\sigma)}\\
        &\lesssim b^{n} r^{-1/2} \left(b^{-n} \sum_{I \in \D_n} \mu_n(I)^p \right)^{1/p}\\
        &\lesssim (b^{n(1+\varepsilon_{p,1,n}-1/p - \widetilde{\tau}(p)/p )})r^{-1/2}.
    \end{align*}    
    We sum the above estimates over $0 \leq j \leq n$ and use that $r\geq b^{2n}$, $\widetilde{\tau}(p) \leq p-1$. This implies
    \begin{align*}
        \sigma_p(\mu_n)(r) &\lesssim (b^{n(1+ \varepsilon_{p,1,n}-1/p -\widetilde{\tau}(p)/p )} ) ( n b^n r^{-1}+ r^{-1/2}) \\
        &\lesssim_\beta r^{(\varepsilon_{p,1,n}-\beta)/2}\,,
    \end{align*}
    as required.
\end{proof}

\begin{lemma}\label{lem:circ_decay_interm}
    Let $b^n \leq r \leq b^{2n}$. For any $\delta>0$, almost surely we have
    $$
    \sigma_p (\mu_{n+1} - \mu_n)(r) \lesssim b^{-n (\widetilde{\tau}(2) - (1+\widetilde{\tau}(p))/p -3\delta) }  r^{\widetilde{\tau}(2)/2 - (1+\widetilde{\tau}(p))/p}.
    $$
\end{lemma}
\begin{proof}
    Let $\delta>0$. Let $1 \leq k \leq n$ such that $b^{n+k-1} \leq r \leq b^{n+k}$ and let $\xi = r\theta$, where $\theta\in S^1$. Note that $\mathcal{S}_\ell(\xi) = \mathcal{S}_\ell(\theta)$. Lemma \ref{lem:conc_app} implies that almost surely we have
    \begin{align*}
        &\sigma_p(\mu_{n+1} - \mu_n)(r) \\
        &\lesssim b^{2n\delta}\sum_{1 \leq \ell\leq k} b^{\ell-k}  \norm{\left(\sum_{\iii\in\Lambda^n,D_\iii \subseteq (\mathcal{S}_\ell(\theta))_{b^{1-n}}} \mu_n(D_\iii)^2 \right)^{1/2}}_{L^p(d\sigma)} +b^{-n}\\
        &\lesssim b^{2n\delta}\sum_{1 \leq \ell \leq k} b^{\ell-k}  \left(b^{-\ell} \sum_{I \in \D_\ell} \left| \sum_{\iii\in\Lambda^n,D_\iii \subseteq (\mathcal{S}_\ell(\theta))_{b^{1-n}}} \mu_n(D_\iii)^2 \right|^{p/2} \right)^{1/p}+b^{-n}\\
        &\lesssim b^{2n\delta}\sum_{1 \leq \ell \leq k} b^{\ell-k}  (b^{-\ell(1+\widetilde{\tau}(p))/p - (n-\ell)\widetilde{\tau}(2)/2})b^{n \varepsilon_{p,2,n}}+b^{-n}\,,
    \end{align*}
    where, in the second inequality, we have again used that each $I\in\D_\ell$ belongs to $(\mathcal{S}_\ell(\theta))_{b^{1-n}}$ for $\theta$ in an interval of length $\sim b^{-\ell}$.
    For sufficiently large $n$, Lemma \ref{lem:unif_pq} implies that $\varepsilon_{p,2,n} < \delta$. Since $r \sim b^{n+k}$, and $1-(1-\widetilde{\tau}(p))/p+\widetilde{\tau}(2)/2>0$, the desired result follows by summing over $\ell$.
\end{proof}
Now, we can prove Theorem \ref{thm:circ_decay}.
\begin{proof}[Proof of Theorem \ref{thm:circ_decay}]
    Consider $r$ with $b^{2j} \leq r \leq b^{2j+2}$, $j \in \N$ and let $m \geq 2j$. Lemma \ref{lem:small_f} implies that if $\beta < \widetilde{\tau}(2)$, then almost surely, 
    $$
    \sum_{n=2j}^m \sigma_p(\mu_{n+1} - \mu_n) (r) = \sum_{n=2j}^m b^{-n\beta/2} \lesssim r^{-\beta/2}\,.
    $$
    where the implicit constant is independent of $n$ and $k$.
    Also, Lemmas \ref{lem:unif_pq} and \ref{lem:circ_decay_large} imply that if $\beta < (1+\widetilde{\tau}(p))/p$, then $\sigma_p (\mu_j) (r) \lesssim r^{-\beta/2}$ almost surely. \\
    If $j \leq n \leq 2j$, then Lemma \ref{lem:circ_decay_interm} imply that almost surely we have
    \begin{align*}
    \sum_{n=j}^{2j} \sigma_p(\mu_{n+1}-\mu_n) (r) &\lesssim  \sum_{n=j}^{2j} b^{-n (\widetilde{\tau}(2) -(1+\widetilde{\tau}(p))/p -3\delta)}  r^{\widetilde{\tau}(2)/2 -(1+\widetilde{\tau}(p))/p}\\
    &\lesssim r^{-\widetilde{\tau}(2)/2 + 3\delta} + r^{-(1+\widetilde{\tau}(p))/2p + 3\delta/2}.
    \end{align*}
    Combining the estimates above, we obtain that for $\beta < \min \{\widetilde{\tau}(2), (1+\widetilde{\tau}(p))/p \}$,
    \[
    \begin{split}
        \sigma_p(\mu_m)(r) &\leq \sigma_p(\mu_j)(r) +  \sum_{n=j}^{2j}  \sigma_p(\mu_{n+1} - \mu_n)(r)+ \sum_{n=2j}^{m}  \sigma_p(\mu_{n+1} - \mu_n)(r)  \\
        &\lesssim r^{-\beta/2}\,,
    \end{split}
    \]
    completing the proof.
\end{proof}

\section{Appendix}

Let $\nu$ be the Mandelbrot cascade on $[0,1]^d$ corresponding to a random variable $W$ satisfying \eqref{eq:unif} and \eqref{eq:moments}. The following theorem generalizes the main result of \cite{CLS}. The result, as stated here, is a special case of more general results obtained recently in \cite{CHQW,LinQiuTan2024,LinQiuTan2025}. Compared to these papers, our approach below provides a simple proof under the fairly general moment condition \eqref{eq:moments}.
\begin{theorem}\label{thm:main}
    $\dim_F \nu=\min\{2,\dim_2 \nu\}$ almost surely on non-extinction.
\end{theorem}
Recall that $\dim_F \eta \leq \dim_2 \eta$ holds for any finite Borel measure $\eta$. The upper bound $\dim_F\nu\le 2$ is proved in \cite[Theorem 3.2]{CLS} in the case of iid $W_i$, but the proof extends to general cascades without any difficulty. It thus remains to prove the lower bound $\dim_F\nu\ge\min\{2,\dim_2 \nu\}$. 
We first recall the following lemma, which reveals, in particular, that $\widetilde{\tau}(2)$ is almost surely the same as $\dim_2 \nu$.
\begin{lemma}\label{lem:dim_2}
Almost surely on non-extinction,
\[\dim_2(\nu)=\lim_{n\to\infty}\frac{\log\sum_{I\in\Q_n}\nu_n(I)^2}{-n}=\widetilde{\tau}(2)\,.
\]
\end{lemma}

\begin{proof}
For all $q>1$, we define auxiliary random variables
\begin{equation*}
    W_q=\frac{b^d(W_j^q)_{j\in\Lambda}}{\sum_{j\in\Lambda}\E(W_j^q)}\,.
\end{equation*}
Note that each $W_q$ satisfies \eqref{eq:cascade_gen}. Denote
\[Y_n(q):=b^{ndq}\left(\sum_{j\in\Lambda}\E(W_j^q)\right)^{-n}\sum_{I\in\Q_n}\nu_n(I)^q\,.\]
Then $Y_n$ is the total measure at generation $n$ of the cascade corresponding to $W_q$. If $W_q$ is subcritical, $Y_n(q)$ converges to a nonzero limit almost surely, on non-extinction, and whence
\begin{equation}\label{eq:dim_qn}
\lim_{n\to\infty}\frac{\log\sum_{I\in\Q_n}\nu_n(I)^q}{-n}=\tau(q)\,.   
\end{equation}
After this, the proof proceeds verbatim with \cite[Lemma 2.2]{CLS}.
\end{proof}

The heart of the proof of Theorem \ref{thm:main} is the following application of Lemma \ref{lem:conc}. 

\begin{lemma}\label{main_probab_lemma}
    Let $\beta<\widetilde{\tau}(2)$. Then, almost surely, 
    \begin{equation}\label{eq:small_freq} 
|\widehat{\nu_{n+1}}(\xi)-\widehat{\nu_n}(\xi)|\lesssim_\beta b^{-n\beta/2}
\end{equation}
holds for all $n\in\N$ and all $\xi\in\R^d$, $|\xi|_\infty\le b^{n+1}$.  Here, the implicit constant is random, but independent of $n$ and $\xi$.
\end{lemma}
\begin{proof} 
Let $\delta>0$. Given $\xi$ and $n$, we may write
$$
\widehat{\nu_{n+1}}(\xi)-\widehat{\nu_{n}}(\xi)=\sum_{\iii\in\Lambda^n} \nu_n(Q_\iii)X_{\iii,\xi}\,,
$$
where
\begin{align*}
X_{\iii,\xi}=b^{nd}\sum_{j\in\Lambda}\left((W_\iii)_j-1\right)\int_{Q_{\iii,j}}\exp(-2\pi i x\cdot\xi)\,dx\,,
\end{align*}
and $W_\iii$ is an independent copy of $W$ and $Q_{\iii,j} = x_{(\iii, j)} + [0,b^{-(n+1)})^d$ with $(\iii,j) = (i_1, i_2, \cdots, i_n, j) $ for each $\iii \in \Lambda^n$ and $j \in \Lambda$. We note that, conditional on $\mathcal{F}_n$ (the sigma-algebra generated by $W_\iii$, $\iii\in\cup_{k=0}^n\Lambda^k$), the random variables $X_{\iii,\xi}$ are independent, have zero mean, and satisfy 
\[\PP\left(|re X_{\iii,\xi}|>t\right),\PP\left(|im X_{\iii,\xi}|>t\right)\le\Phi(t):=\PP\left(\max_{j\in\Lambda} W_j>t+1\right)\,.\]
Note that \eqref{eq:moments} implies $\Phi(t) \lesssim_p t^{-p}$ for all $0<p<\infty$. 

For each $n \in \N$, denote $S_n = \sum_{Q_\iii \in \Q_n} \nu_n^2 (Q_\iii)$. We may then use the Lemma \ref{lem:conc} for the real and imaginary parts of $X_{\iii,\xi}$ with the choice $N=b^{nd}$, $q=1$, $t=b^{2n\delta} S_n^{1/2} $ and $\lambda=t^{-1} b^{n\delta}$. Then, the lemma yields
\begin{align*}
   \PP\left(\left|\widehat{\nu_{n+1}}(\xi)-\widehat{\nu_{n}}(\xi)\right|> b^{2n\delta} S_n^{1/2} \right)=2b^{nd}\Phi(M)+2\exp(-b^{n\delta }+O(b^{-2n \delta}))\,,
\end{align*}
where
\[\frac{M}{\log M}=\frac{1}{\lambda\max_{\iii\in\Lambda^n}\nu_n(Q)}\ge b^{n\delta}\]
and whence we choose $M\ge b^{n\delta}$. Combining with a union bound,
\begin{equation*}
    \begin{split}
    &\PP\left(\left|\widehat{\nu_{n+1}}(\xi)-\widehat{\nu_{n}}(\xi)\right|>b^{2n\delta} S_n^{1/2} \text{ for some }\xi\in b^{-nd/2}\Z^d\,,\, |\xi|_\infty \le b^{n+1}\right)\\
    &\lesssim b^{n(d^2/2+d)}\left(b^{nd} \Phi(b^{n\delta})+\exp(-b^{n\delta })\right).
    \end{split}
\end{equation*}
Using that $\Phi(t)\lesssim t^{-p}$ for all $p$ and choosing $p>(d^2/2+2d)/\delta$, we have
\begin{equation*}
    \PP\left(\left|\widehat{\nu_{n+1}}(\xi)-\widehat{\nu_{n}}(\xi)\right|>b^{2n\delta} S_n^{1/2} \text{ for some }\xi\in b^{-nd/2}\Z^d\,,\, |\xi|_\infty\le b^{n+1}\right)\lesssim_\delta b^{-nc}\,,
\end{equation*}
for some $c=c_\delta>0$.

Applying the Borel-Cantelli lemma, it follows that, almost surely,
\begin{equation*}
    |\widehat{\nu_{n+1}}(\xi)-\widehat{\nu_{n}}(\xi)|\le b^{2n\delta} S_n^{1/2}
\end{equation*}
for all $\xi\in b^{-nd/2}\Z^d$, $|\xi|_\infty \leq b^{n+1}$ and for all except finitely many $n$. 

Since $\nu_n([0,1]^d)$ is a martingale with $\E(\nu_n([0,1]^d)) =1$, we have $\sup_n \nu_n ( [0,1]^d)< \infty$ almost surely. Therefore, $\widehat{\nu_n}(\xi)$ is Lipschitz with a Lipschitz constant independent of $n$. Thus, we observe that, almost surely,
$$
    |\widehat{\nu_{n+1}}(\xi)-\widehat{\nu_{n}}(\xi)|\lesssim b^{2n\delta} S_n^{1/2} +b^{-nd/2}
$$
for all $|\xi|_\infty \leq b^{n+1}$, and for all except finitely many $n$.

Combining with Lemma \ref{lem:dim_2} and the fact that $\widetilde{\tau}(2)\leq d$, this completes the proof.
\end{proof}

\begin{remark}\label{rem:appendix}
We note that the above proof works, verbatim, for the curvilinear cascades $\mu_n$. In this case, $d=1$ and the integrals in the definition of $X_{\iii,j}$ are over $\gamma(Q_{\iii,j})$. Thus, we have also verified the Lemma \ref{lem:small_f}. 
\end{remark}

\begin{proof}[Proof of Theorem \ref{thm:main}]
Once we have Lemma \ref{main_probab_lemma} at our disposal, the proof is identical to the proof of \cite[Theorem 3.1]{CLS} (see also the proof of \cite[Theorem 14.1]{SS}). We provide the details for the convenience of the reader. Let $\beta<\min\{2,\widetilde{\tau}(2)\}$. Almost surely, there is a finite implicit constant such that the claim of Lemma \ref{main_probab_lemma} holds. Conditional on this, we compute as follows (once $\beta$ and the implicit $O(1)$-constant of Lemma \ref{main_probab_lemma} are fixed, the proof does not contain any probabilistic elements):

If $|\xi|\le b^{n+1}$, we may use \eqref{eq:small_freq}.

If $|\xi| >b^{n+1}$, we write $\xi=b^{n+1}q+\xi'$ where $q\in\Z^d$ and $|\xi'|<b^{n+1}$. Using change of variables in each coordinate (this is applied to compute $\widehat{\textbf{1}_{Q_\iii}}(\xi)$ for $\iii\in\Lambda^\iii$) gives the identity
\begin{align*}
    \widehat{\nu_{n+1}}(\xi)-\widehat{\nu_n}(\xi)=\left(\widehat{\nu_{n+1}}(\xi')-\widehat{\nu_n}(\xi')\right)\prod_{1\le j\le d\,,\,\xi_j\neq 0}\frac{|\xi'_j|}{|\xi_j|}\,.
\end{align*}
Each term in the above product is $\le 1$ and the smallest term is $\lesssim b^{n+1}/|\xi|_\infty$. Combining this with \eqref{eq:small_freq} (applied to $\xi'$), this gives
\begin{equation}\label{eq:large_freq}
|\widehat{\nu_{n+1}}(\xi)-\widehat{\nu_{n}}(\xi)|\le  \frac{b^{n+1}}{|\xi|_\infty}  \left|\widehat{\nu_{n+1}}(\xi')-\widehat{\nu_{n}}(\xi')\right|\le\frac{b^{n(1-\beta/2)}}{|\xi|_\infty} \,.
\end{equation}

We may now finish the proof as follows. If $\xi\in \R^d$, $|\xi|_\infty>b$, let $n_\xi\in\mathbb{N}$ so that $b^{n_\xi}<|\xi|_\infty\le b^{n_\xi+1}$. Using \eqref{eq:large_freq} for $n\le n_\xi$ and \eqref{eq:small_freq} for $n>n_\xi$ and telescoping, we have
\begin{align*}
|\widehat{\nu_{m}}(\xi)-\widehat{\nu_0}(\xi)|&\le\sum_{n=0}^{n_\xi}|\widehat{\nu_{n+1}}(\xi)-\widehat{\nu_n}(\xi)|+\sum_{n=n_\xi+1}^{m-1}|\widehat{\nu_{n+1}}(\xi)-\widehat{\nu_n}(\xi)|\\
&\lesssim |\xi|^{-1}\sum_{n=0}^{n_\xi}b^{n(1-\beta/2)}+\sum_{n=n_\xi+1}^{m-1}b^{-n\beta/2}\\
&\lesssim |\xi|^{-\beta/2}\,,
\end{align*}
for all $m\ge n_\xi$.
Since also $\widehat{\nu_0}(\xi)\lesssim |\xi|^{-1} \lesssim_\beta |\xi|^{-\beta/2}$, we get
\[
\widehat{\nu}(\xi)=\lim_{m\rightarrow \infty} \widehat{\nu_m}(\xi)\lesssim_\beta |\xi|^{-\beta/2} + |\xi|^{-1} \lesssim |\xi|^{-\beta/2}\,,
\]
where the implicit constant is independent of $\xi$. Since $\beta < \min \{2, \widetilde{\tau}(2)\} $ is arbitrary, together with Lemma \ref{lem:dim_2}, this implies that $\dim_F \nu \geq \min\{2, \dim_2 \nu\}$ almost surely on non-extinction.
\end{proof}

\bibliographystyle{abbrv}
\bibliography{ref}

\begin{thebibliography}{10}

\bibitem{AK}
S.~Armstrong and T.~Kuusi.
\newblock Renormalization group and elliptic homogenization in high contrast.
\newblock {\em Invent. Math.}, 242(3):895--1086, 2025.

\bibitem{Barral2000}
J.~Barral.
\newblock Continuity of the multifractal spectrum of a random statistically self-similar measure.
\newblock {\em J. Theoret. Probab.}, 13(4):1027--1060, 2000.

\bibitem{CLS}
C.~Chen, B.~Li, and V.~Suomala.
\newblock Fourier dimension of {M}andelbrot multiplicative cascades.
\newblock {\em Comm. Math. Phys.}, 406(8):Paper No. 182, 15, 2025.

\bibitem{CHQW}
X.~Chen, Y.~Han, Y.~Qiu, and Z.~Wang.
\newblock Harmonic analysis of {M}andelbrot cascades--in the context of vector-valued martingales.
\newblock Preprint, available at arXiv:2409.13164.

\bibitem{FalconerJin2019}
K.~Falconer and X.~Jin.
\newblock Exact dimensionality and projection properties of {G}aussian multiplicative chaos measures.
\newblock {\em Trans. Amer. Math. Soc.}, 372(4):2921--2957, 2019.

\bibitem{GarbanVargas}
C.~Garban and V.~Vargas.
\newblock Harmonic analysis of {G}aussian multiplicative chaos on the circle.
\newblock {\em Probab. Theory Related Fields}, 2026.

\bibitem{Heur}
Y.~Heurteaux.
\newblock An introduction to {M}andelbrot cascades.
\newblock In {\em New trends in applied harmonic analysis}, Appl. Numer. Harmon. Anal., pages 67--105. Birkh\"auser/Springer, Cham, 2016.

\bibitem{HuntKaloshin}
B.~R. Hunt and V.~Y. Kaloshin.
\newblock How projections affect the dimension spectrum of fractal measures.
\newblock {\em Nonlinearity}, 10(5):1031--1046, 1997.

\bibitem{Kahane1993}
J.-P. Kahane.
\newblock Fractals and random measures.
\newblock {\em Bull. Sci. Math.}, 117(1):153--159, 1993.

\bibitem{KahanePeyriere1976}
J.-P. Kahane and J.~Peyri\`ere.
\newblock Sur certaines martingales de {B}enoit {M}andelbrot.
\newblock {\em Advances in Math.}, 22(2):131--145, 1976.

\bibitem{Kaufman1981}
R.~Kaufman.
\newblock On the theorem of {J}arn\'ik and {B}esicovitch.
\newblock {\em Acta Arith.}, 39(3):265--267, 1981.

\bibitem{LinQiuTan2024}
Z.~Lin, Y.~Qiu, and M.~Tan.
\newblock Harmonic analysis of multiplicative chaos {P}art {I}: the proof of {G}arban-{V}argas conjecture for 1{D} {GMC}.
\newblock Preprint, available at arXiv:2411:13923.

\bibitem{LinQiuTan2025}
Z.~Lin, Y.~Qiu, and M.~Tan.
\newblock Harmonic analysis of multiplicative chaos {P}art {II}: a unified approach to {F}ourier dimensions.
\newblock Preprint, available at arXiv:2505.03298.

\bibitem{Lyons1995}
R.~Lyons.
\newblock Seventy years of {R}ajchman measures.
\newblock In {\em Proceedings of the {C}onference in {H}onor of {J}ean-{P}ierre {K}ahane ({O}rsay, 1993)}, pages 363--377, 1995.

\bibitem{Mandelbrot1974}
B.~Mandelbrot.
\newblock Intermittent turbulence in self similar cascades: divergence of high moments and dimension of carrier.
\newblock {\em J. Fluid Mech.}, 62:331--333, 1974.

\bibitem{Mandelbrot1976}
B.~Mandelbrot.
\newblock Intermittent turbulence and fractal dimension: kurtosis and the spectral exponent {$5/3+B$}.
\newblock In {\em Turbulence and {N}avier-{S}tokes equations ({P}roc. {C}onf., {U}niv. {P}aris-{S}ud, {O}rsay, 1975)}, volume Vol. 565 of {\em Lecture Notes in Math.}, pages 121--145. Springer, Berlin-New York, 1976.

\bibitem{Mattila2015}
P.~Mattila.
\newblock {\em Fourier analysis and {H}ausdorff dimension}, volume 150 of {\em Cambridge Studies in Advanced Mathematics}.
\newblock Cambridge University Press, Cambridge, 2015.

\bibitem{Molchan1996}
G.~M. Molchan.
\newblock Scaling exponents and multifractal dimensions for independent random cascades.
\newblock {\em Comm. Math. Phys.}, 179(3):681--702, 1996.

\bibitem{Ryou}
D.~Ryou.
\newblock Near-optimal restriction estimates for {C}antor sets on the parabola.
\newblock {\em Int. Math. Res. Not. IMRN}, (6):5050--5099, 2024.

\bibitem{SS}
P.~Shmerkin and V.~Suomala.
\newblock Spatially independent martingales, intersections, and applications.
\newblock {\em Mem. Amer. Math. Soc.}, 251(1195):v+102, 2018.

\end{thebibliography}

\end{document}